\documentclass{amsart}

\usepackage{amsthm,amsmath,amsfonts,amssymb}
\usepackage{mathrsfs,latexsym,enumitem}
\usepackage{graphicx,caption,subcaption}
\usepackage{tikz}
\usepackage{verbatim}

\newtheorem{theorem}{Theorem}
\newtheorem{lemma}[theorem]{Lemma}
\newtheorem{proposition}[theorem]{Proposition}

\newtheorem{remark}[theorem]{Remark}

\numberwithin{equation}{section}

\newcommand\lbd{\underline{\mbox{\rm dim}}_{\it B}\,} 
\newcommand\bi{{\bf i}}
\newcommand\ri{{\rm i}}
\newcommand\E{\mathbb{E}}
\renewcommand\P{\mathbb{P}}

\begin{document}

\title{Dimension conservation for self-similar sets and fractal percolation} 

\author{Kenneth Falconer}
\address{Mathematical Institute, University of St Andrews, North Haugh, St Andrews, Fife, KY16 9SS, United Kingdom}              
\email{kjf@st-andrews.ac.uk}

\author{Xiong Jin}
\address{School of Mathematics, Alan Turing Building, The University of Manchester,
Manchester, M13 9PL, United Kingdom
}
\email{xiong.jin@manchester.ac.uk}

\begin{abstract}
We introduce a technique that uses projection properties of fractal percolation to establish dimension conservation results for sections of deterministic self-similar sets. For example, let $K$ be a self-similar subset of $\mathbb{R}^2$ with Hausdorff dimension $\dim_H K >1$ such that the rotational components of the underlying similarities generate the full rotation group. Then for all $\epsilon >0$, writing $\pi_\theta$ for projection onto the line $L_\theta$ in direction $\theta$, the Hausdorff dimensions of the sections satisfy $\dim_H (K\cap \pi_\theta^{-1}x)> \dim_H K - 1 - \epsilon$  for a set of $x \in L_\theta$ of positive Lebesgue measure, for all directions $\theta$ except for those in a set of Hausdorff dimension 0. For a class of self-similar sets we obtain a similar conclusion for all directions, but with lower box dimension replacing Hausdorff dimensions of sections. We obtain similar inequalities for the dimensions of sections of Mandelbrot percolation sets.
\end{abstract}

\maketitle

\section{Introduction}
\setcounter{theorem}{0}

Relating the Hausdorff dimension $\dim_H K$ of a set $K \subset \mathbb{R}^d$ to the dimensions of its sections and projections has a long history. The best-known result on projections is that, if $K$ is Borel or analytic, then, writing 
$\pi_V: \mathbb{R}^d \to V$ for  orthogonal projection onto the subspace $V$, \begin{equation}\label{dimproj}
\dim_H \pi_V K=\min(k,\dim_H K),
\end{equation}
 for almost all $k$-dimensional subspaces $V$ (with respect to the natural invariant measure on subspaces). For sections of sets, for almost all $k$-dimensional subspaces $V$, the dimensions of the {\it sections} or {\it slices}  $\pi_V^{-1}x \cap K$ of $K$ satisfy
\begin{equation*}\label{dimsec}
\dim_H (K \cap \pi_V^{-1}x) \leq \max(0,\dim_H K - k)
\end{equation*}
 for Lebesgue almost all $x \in V$ (we take $\dim_H \emptyset =-\infty$).
  Moreover, for all $\epsilon>0$ and almost all $V$, there is a set $W_\epsilon \subset V$ of positive $k$-dimensional Lebesgue measure such that
\begin{equation}\label{dimsec1}
\dim_H (K \cap \pi_V^{-1}x) \geq \max(0,\dim_H K - k) - \epsilon
\end{equation}
 for $x\in W_\epsilon$.
These inequalities were obtained by Marstrand \cite{Mar54} for subsets of the plane, and extended to general $d$ and $k$ by Mattila \cite{Mat75}. Kaufman \cite{Kau68} introduced the potential theoretic method which is now commonly used in studying dimensions of projections and sections of sets.
 
 These properties are  complemented by the fact \cite{Mar54a} that, for \textit{all} $k$-dimensional subspaces $V$, for all $0 \leq \Delta \leq d-k$,
\begin{equation*}\label{dimcons}
\Delta + \dim_H \{x \in V: \dim_H (K\cap \pi_V^{-1}x) \geq \Delta\} \leq \dim_H K.
\end{equation*}
In particular, if $\dim_H K> k$ then for all $V$
$$\dim_H (K\cap \pi_V^{-1}x)\leq \dim_H K - k$$
for Lebesgue almost all $x\in V$.
A good exposition of this material may be found in \cite{Mat95}.

Fursternberg \cite{Fur} introduced the notion of dimension conservation: given $K \subset \mathbb{R}^d$, a projection $\pi_V$ is said to be \textit{dimension conserving for} $K$ if there is a number $\Delta >0$ such that 
\begin{equation}
\Delta + \dim_H \{x \in V: \dim_H (K\cap \pi_V^{-1}x) \geq \Delta\} \geq \dim_H K \label{dimcon} 
\end{equation} 
In this paper we consider a slightly weaker property when $\dim_H K > k$. We say that a projection $\pi_V$ is {\it weakly dimension conserving} if, for all $\epsilon >0$,
\begin{equation}\label{weakdimcon}
\dim_H (K\cap \pi_V^{-1}x)> \dim_H K - k - \epsilon\quad \mbox{ for all } x\in W,
\end{equation}
where $W$ is a `large' subset of $V$, either with $\dim_H W = k$ or with ${\mathcal L}^k(W) >0$, where  ${\mathcal L}^k$ denotes $k$-dimensional Lebesgue measure.
It follows from (\ref{dimsec1}) that $\pi_V$ is weakly dimension conserving for almost every $k$-dimensional subspace $V$.

There has been great interest recently in identifying classes of sets, in particular classes of self-similar sets and their variants, for which these various inequalities
hold for {\it all}, rather than just almost all, subspaces. Several papers establish (\ref{dimproj}) for all  projections for classes of self-similar sets \cite{Fur,HoSh12,PS,SS} and for percolation on self-similar sets \cite{FJ,RS,RS2,RS3,SV}. Here we consider dimensions of sections, and identify sets for which (\ref{weakdimcon}), or a similar inequality for box-counting dimension, holds for all subspaces $V$.

Recall that an {\it iterated function system} (IFS) $\mathcal{I}=\{f_i\}_{i=1}^{m}$ on $\mathbb{R}^d$ is a family of $2 \leq m < \infty$ contractions $f_i : \mathbb{R}^d\to \mathbb{R}^d$. An IFS determines a unique non-empty compact  $K \subset \mathbb{R}^d$ such that 
\begin{equation}\label{attractor}
K= \bigcup_{i=1}^m f_i(K),
\end{equation}
called the {\it attractor} of the IFS, see \cite{Falconer14,Hut81}. If the $f_i$ are all similarities then $K$ is {\it self-similar}.
The IFS satisfies the {\it strong separation  condition} (SSC) if the union (\ref{attractor}) is disjoint, and the {\it open set condition} (OSC) if there is a non-empty open set $U$ such that $\cup_{i=1}^m f_i(U)\subset U$ with this union disjoint. If either SSC or OSC hold then $\dim_H K = s$ where $s$ is given by $\sum_{i=1}^m r_i^s=1$, where $r_i$ is the similarity ratio of $f_i$. 

We may write an IFS of (orientation preserving)  similarities as 
\begin{equation*}\label{IFS}
 \mathcal{I}=\{f_i=r_iR_i\cdot+a_i\}_{i=1}^{m}
\end{equation*}
where $R_i \in SO(d,\mathbb{R})$ is a rotation, $r_i$ is the scaling ratio and $a_i$ is a translation. If the group $G$ generated by $\{R_1,\ldots,R_m\}$ is dense in $SO(d,\mathbb{R})$ we say that the IFS has {\it dense rotations}.

A number of results on dimension conservation of self-similar sets have been established.
Furstenberg \cite{Fur} showed that (\ref{dimcon}) holds for projections onto all subspaces $V$ for a class of  `homogeneous' sets. These include self-similar sets where the IFS 
$\mathcal{I}$ consists of contracting homotheties (i.e. similarities without rotation or reflection so that $R_i = I$ for all $i$) that satisfy SSC or OSC. For example, variants on the Sierpi\'{n}ski carpet are of this type, where the value of $\Delta$ in (\ref{dimcon}) depends on  the subspace $V$. There are detailed analyses of sections of the Sierpi\'{n}ski carpet  in \cite{MS,LXZ} and of sections of the Sierpi\'{n}ski gasket or triangle in \cite{BFS}.
In the case where the IFS ${\mathcal I}$ satisfies OSC and the group  generated by $\{R_1,\ldots,R_m\}$ is finite, then every projection is dimension conserving, that is for all $V$ (\ref{dimcon}) holds for some number $\Delta$, see \cite{FJ,Ho12}.

In this paper we demonstrate that many self-similar sets $K$ are weakly dimension conserving for  all, or virtually all, projections $\pi_V$. For self-similar sets in $\mathbb{R}^2$ where  ${\mathcal I}$ satisfies OSC and has dense rotations and $\dim_HK > 1$, (\ref{weakdimcon}) holds with $\mathcal{L}(W)>0$ for all $\epsilon>0$ and for projections onto all lines $V$, except for lines in a set of directions of Hausdorff dimension 0  (Theorem \ref{thmhdgen}). Provided that we replace Hausdorff dimension by lower box dimension  on the left-hand side of the inequality  we get (\ref{weakdimcon}) for all lines, for a large class of sets that satisfy a projection condition   (Theorem \ref{dimconthm}).
We also show that, almost surely,  (\ref{weakdimcon}) is true for all $k$-dimensional subspaces $V$ for random subsets  of $\mathbb{R}^d$ obtained by the Mandelbrot percolation process  (Theorem \ref{mandper}).

The idea is to demonstrate weak dimension conservation for a deterministic set $K$ by running a percolation-type process on $K$ to `probe' the dimensions of its sections. We construct   random sets $K^\omega\subset K$ such that
$k<\dim K^\omega < k+\epsilon/2$ with positive probability. Writing $L_x$ for the $(d-k)$-plane through $x$ and perpendicular to $V$, if $\dim (K\cap L_x) < \dim K - k -\epsilon$ for some $x \in V$ there is a high probability  that  $K^\omega\cap L_x =\emptyset$ or equivalently that $x\not \in \pi_V K^\omega$. By invoking results on projections of random sets that show that with positive probability $\dim \pi_V K^\omega = k$, we conclude that there must be a significant subset of $x \in V$, indeed a subset of dimension $k$, for which this does not occur.

We formulate this principle in a general context in Propositions \ref{prop1} and \ref{prop2}. To apply it in various settings we utilise results on dimensions of projections of percolation sets from \cite{FJ,RS,RS2,RS3}. Theorems \ref{thmhd} and \ref{thmhdgen} depend on the absolute continuity of projections of an alternative type of random measure, and this is established in Theorem \ref{absss} which is a random version of a deterministic result of Shmerkin and Solomyak \cite{SS}.

The authors are grateful to Mike Hochman for comments on an earlier version of this paper.

\section{Estimates for dimensions of sections using random subsets}\label{sec2}
\setcounter{theorem}{0}

In this section we present a general formulation of our method for obtaining lower bounds for the dimensions of sections of a set given a knowledge of the dimensions of projections of related random subsets.  
The method applies to sets that can be modeled in terms of an infinite rooted tree. These include self-similar sets, where the tree provides a natural description of the hierarchical construction of the set, but extends to a many further fractals.

Let $\Lambda=\{1,\ldots,m\}$ be an alphabet of $m\ge 2$ symbols, with  $ \Lambda^n$  denoting the set of words of length $n \geq 0$. Let $\Sigma_*: =\cup_{n\ge 0} \Lambda^n$ be the set of finite words and 
$\Sigma :=\Lambda^\mathbb{N}$ the corresponding {\it symbolic space} of all infinite words. For each $\mathbf{i}\in \Sigma_*$ denote by $[\mathbf{i}]\subset \Sigma$ the set of infinite words that start with $\mathbf{i}$, that is the {\it cylinder} rooted at $\mathbf{i}$. We denote the diameter of a set $A\subset \mathbb{R}^d$ by $|A|$.

We consider fractals which are the image of a subset of symbolic space under a  continuous mapping $\Phi:\Sigma \mapsto \mathbb{R}^d$ with the usual metrics. For each $\mathbf{i}\in \Sigma_*$ we write $B(\Phi[\mathbf{i}])$ for the closed convex hull of $\Phi[\mathbf{i}]$. We also assume throughout that there is a number
$d_0>0$ such that 
\begin{equation*}\label{fat}
\frac{\mbox{inradius } B(\Phi[\mathbf{i}])}{\mbox{diameter } B(\Phi[\mathbf{i}])}\geq d_0 \quad
\mbox{ for all } \mathbf{i}\in \Sigma_*;
\end{equation*}
thus the convex hulls cannot get `too long and thin'.
We assume throughout that $\Phi$ satisfies  the following two conditions:
\medskip

\begin{itemize}
\item[(1)] There exist $0<c_0, c_1<\infty$ such that for all $\rho\in(0,c_0)$, the set
\begin{equation}\label{lambdar}
\Lambda_\rho=\{\mathbf{i}\in \Sigma_*:  \rho\le |\Phi[\mathbf{i}]|< c_1\rho\}
\end{equation}
yields a finite covering of $\Sigma$, that is $\#\Lambda_\rho<\infty$ and $\Sigma=\cup_{\mathbf{i}\in \Lambda_\rho} [\mathbf{i}]$; 
\medskip

\item[(2)] There exists an integer $n_0$ such that for all $\rho\in(0,c_0)$ and $x\in \mathbb{R}^n$,
\begin{equation}\label{overlap}
\#\{\mathbf{i}\in \Lambda_\rho: x\in B(\Phi[\mathbf{i}])\}\le n_0.
\end{equation}
\end{itemize}
These conditions will certainly be satisfied if $\Phi$ codes the attractor of an IFS satisfying OSC.
 
We may define measures  of Hausdorff type on subsets of $\Phi(\Sigma)$ by setting, for all $s>0$, $F \subset \Phi(\Sigma)$ and $\delta >0$,
\begin{equation}\label{mesdel}
{\mathcal M}^s_\delta (F) = \inf\Big\{\sum_{j=1}^\infty\big|\Phi[\bi_j]\big|^s :\Phi^{-1}(F) \subset 
\bigcup_{j=1}^\infty [\bi_j],\  \big|\Phi[\bi_j]\big|\leq \delta \Big\}
\end{equation}
and
\begin{equation*}
{\mathcal M}^s (F) =\lim_{\delta\searrow 0} {\mathcal M}^s_\delta (F).
\end{equation*}
Then ${\mathcal M}^s$ is equivalent to the restriction of $s$-dimensional Hausdorff measure ${\mathcal M}^s$ to $\Phi(\Sigma)$. Clearly ${\mathcal H}^s(F) \leq {\mathcal M}^s(F)$ for $F \subset \Phi(\Sigma)$. For the opposite inequality (to within a constant multiple), note that  the number of sets $\Phi[\bi]$ with $\bi \in \Lambda_\rho$ that overlap $U\cap \Phi(\Sigma)$ is bounded for all $U \subset \mathbb{R}^n$ with $|U| = \rho<c_0$, from comparing the volumes of maximal inscribed balls  of $B(\Phi[\mathbf{i}])$ with that of some ball centered in $U$ of radius $|U|$ and using (\ref{overlap}). In particular, $\dim_H F= \inf\{s: {\mathcal M}^s (F) = 0\} =  \sup\{s: {\mathcal M}^s (F) = \infty\}$ for $F \subset \Phi(\Sigma)$.

In a similar way,  (\ref{lambdar}) and (\ref{overlap}) imply that the box-counting dimension of subsets of $\Phi(\Sigma)$ may be found by counting cylinders. In particular, the lower box-counting dimension of  $F \subset \Phi(\Sigma)$ is given by 
\begin{equation}\label{lbd}
\lbd F = \varliminf_{\rho \to 0} \frac{\log \{\# \bi \in\Lambda_\rho: F \cap B(\Phi[\bi])\neq 0\}}{-\log \rho}.
\end{equation}
Let $\mathcal{B}_\Sigma$ be the $\sigma$-field generated by the cylinders of $\Sigma$. Let $\mathbb{P}$ be a probability measure on $\mathcal{B}_\Sigma$. Let $\Sigma^\omega$ be a random subset of $\Sigma$ and let
\[
\Sigma_*^\omega : =\{\mathbf{i}\in \Sigma_*: [\mathbf{i}]\cap \Sigma^\omega\neq \emptyset\}.
\]
We adopt the convention that $A^\omega :=A\cap \Sigma^\omega$ if $A$ is a subset of $\Sigma$ and $A^\omega :=A\cap \Sigma_*^\omega$ if $A$ is a subset of $\Sigma_*$.

For $\alpha\ge 0$ we say that $\Sigma^\omega$ is an $\alpha${\it -random subset of  $\Sigma$} if there exists a constant $c_2<\infty$ such that for all $\rho\in(0,c_0)$ and all $\mathbf{i}\in \Lambda_\rho$,
\begin{equation}\label{prob}
\mathbb{P}(\mathbf{i}\in \Lambda^\omega_\rho)\le c_2 \rho^\alpha.
\end{equation}
For our applications, $\Sigma^\omega$ will typically be the symbolic set underlying fractal percolation on $K$, so that  $\Phi (\Sigma^\omega) = K^\omega$. 

Let $V$ be a $k$-dimensional subspace of $\mathbb{R}^d$ and let $\pi_V: \mathbb{R}^d\to V$ denote orthogonal projection onto $V$.  Write ${\mathcal L}^k$ for $k$-dimensional  Lebesgue measure on $V$ identified with $\mathbb{R}^k$ in the obvious way. (If $k=1$ then $V$ is a line and we write ${\mathcal L}$ for   Lebesgue measure on $V$.) 

The following two propositions are our principal tools.  The first, which concerns the Hausdorff measure of sections, has stronger hypotheses on the projection of the random subset but a weaker condition on the projection of the original set, than the second which concerns the lower box dimension of sections.

\begin{proposition}\label{prop1}
Let $A\in \mathcal{B}_\Sigma$. Let $\Sigma^\omega$ be an $\alpha$-random subset of $\Sigma$ for some $\alpha>0$, let $\Phi:\Sigma \to \mathbb{R}^d$ satisfy $(1)$ and $(2)$ above, and let $V$ be a $k$-dimensional subspace of $\mathbb{R}^d$.
If $\mathbb{P}\big({\mathcal L}^k( \pi_V(\Phi(A^\omega)))>0\big)>0$, then 
$$\mathcal{L}^k\big\{x\in V: \dim_H \big(\Phi(A)\cap \pi_V^{-1}(x)\big)\ge \alpha\big\}>0.$$
\end{proposition}

\noindent{\it Proof.}
Let 
$$S = \big\{x \in V:\dim_H \big(\Phi(A)\cap  \pi_V^{-1}(x)\big) < \alpha\big\}.$$
Let $x \in S$. Using (\ref{mesdel}), for all $\epsilon >0$ we may find a set  of words ${\mathcal J}\subset \Sigma_*$ such that 
$\Phi^{-1}\big(\Phi(A)\cap  \pi_V^{-1}(x)\big) \subset  \bigcup_{\bi \in \mathcal J}[\bi]$ and
$\sum_{\bi \in \mathcal J}\big|\Phi[\bi]\big|^\alpha <\epsilon$.
Then  $\Phi(A^\omega)\cap  \pi_V^{-1}(x) \subset \bigcup_{\bi \in {\mathcal J}\cap  \Sigma_*^\omega}\Phi[\bi]$ and
$$\E\big( \#\{\bi \in {\mathcal J}\cap  \Sigma_*^\omega\}\big)
 = \sum_{\bi \in {\mathcal J}}\P\big( \bi \in \Sigma_*^\omega\big) 
 \leq c_2 \sum_{\bi \in {\mathcal J}}\big|\Phi[\bi]\big|^\alpha < c_2 \epsilon,$$
using (\ref{prob}), so $\mathbb{P}\big( \{\bi \in {\mathcal J}\cap  \Sigma_*^\omega\} \neq \emptyset\big) <c_2 \epsilon$. Since $\epsilon$ is arbitrarily small, we conclude that  for all $x \in S$, $\Phi(A^\omega) \cap \pi_V^{-1}(x) = \emptyset$ almost surely. 

By Fubini's theorem, almost surely 
$${\mathcal L}^k \big(S\cap \pi_V(\Phi(A^\omega))\big)={\mathcal L}^k \big(x \in S :\Phi(A^\omega)\cap  \pi_V^{-1}(x) \neq 0\big) = 0.$$ Hence, with positive probability,
$$0<  {\mathcal L}^k \big(\pi_V(\Phi(A^\omega))\big)={\mathcal L}^k \big(\pi_V(\Phi(A^\omega))\setminus S\big) 
\leq {\mathcal L}^k \big(\pi_V(\Phi(A))\setminus S\big).\qed$$  
\medskip

The second general proposition concerns the lower box-counting dimension of sections of sets. Here we require a condition that, for all $\bi\in\Sigma_*$,  the projection of $\Phi [\bi]$ onto the subspace $V$ is the same as that of its convex hull; in particular this will be the case if $\Phi [\bi]$ is connected.

\begin{proposition}\label{prop2}
Let $\Sigma^\omega$ be an $\alpha$-random subset of $\Sigma$ for some $\alpha>0$, let $\Phi:\Sigma \to \mathbb{R}^d$ satisfy $(1)$ and $(2)$ above, and let $V$ be a line, that is a $1$-dimensional subspace of $\mathbb{R}^d$. Suppose that the projection of $\Phi [\bi]$ onto $V$ is the same as that of its convex hull $B(\Phi [\bi])$ for all $\bi\in\Sigma_*$. If $\mathbb{P}(\dim_H \pi_V(\Phi(\Sigma^\omega))=1)>0$, then for every $\epsilon\in(0,\alpha)$, 
$$\dim_H \big\{x\in V:\lbd\big( \Phi(\Sigma)\cap \pi_V^{-1}(x)\big)> \alpha - \epsilon\big\}=1.$$
\end{proposition}

\noindent{\it Proof.}
To keep the notation simple, we give the proof for $\Phi:\Sigma \to \mathbb{R}^2$ where the sections are intersections with lines perpendicular to the line $V$. The proof is virtually identical for $\Phi:\Sigma \to \mathbb{R}^d$ where $d > 2$.
Write $L_x \equiv \pi_V^{-1}(x)$ for the line through $x\in V$ perpendicular to $V$.
For $x \in V$ and $\rho \in(0,c_0)$ write
\begin{equation}\label{nxr}
N(x,\rho) := \# \{\bi \in \Lambda_\rho :B( \Phi[\bi]) \cap L_x\neq \emptyset\}
 \equiv \# \{\bi \in \Lambda_\rho :\Phi[\bi] \cap L_x\neq \emptyset\}
\end{equation}
for the `box counting numbers', where the equivalence follows as  every line that intersects the convex hull $B(\Phi[\bi])$ also intersects $\Phi[\bi]$. 

Here is the first of three subsidiary lemmas within this proof. This enables us to reduce consideration of coverings of subsets of $L_x$ when estimating $N(x,\rho)$ to  a small set of  $x$. We identify  $V$ with $\mathbb{R} \times \{ 0 \} \subset \mathbb{R}^2$ in the obvious way.

\begin{lemma}\label{sub1}
Let $\rho \in(0,c_0)$ and $M>0$. Let $I \subset V$ be an interval with $|I| \leq \rho$ such that $N(x,\rho) \leq M$ for some $x \in I$. Then there exist $x_1,x_2 \in I $ with $x_1 \leq x_2$ such that 
$$N(x_1,\rho), N(x_2,\rho) \leq M$$
and such that, if $x\in I$ has  $N(x,\rho) \leq M$, then, for all $\bi \in \Lambda_\rho$ such that $B( \Phi[\bi]) \cap L_{x} \neq \emptyset$, either $B( \Phi[\bi])\cap L_{x_1} \neq \emptyset$ or $B( \Phi[\bi]) \cap L_{x_2} \neq \emptyset$.
\end{lemma}

\begin{proof} 
Let $x'_{1}= \inf\{x\in I: N(x,\rho) \leq M\}$. If $N(x'_{1},\rho) \leq M$ then take $x_1 = x'_{1}$. Otherwise take $x_1 > x'_{1}$ sufficiently close to $x'_{1}$ to ensure that  both $N(x_{1},\rho) \leq M$ and
\begin{align*}
 \{\bi \in \Lambda_\rho & : B( \Phi[\bi]) \cap L_{x_1}  \neq \emptyset\}\\
& = \{\bi \in \Lambda_\rho : B( \Phi[\bi]) \cap L_{x_1} \neq \emptyset \mbox{ and } \pi_V(\mbox{int}B( \Phi[\bi]))\cap [x_1,\infty) \neq \emptyset\}.
\end{align*}
In the same way, we may take $x_2$ to be $\sup\{x\in I: N(x,\rho) \leq M\}$ or a slightly smaller number if necessary. Clearly we may ensure that $x_1 \leq x_2$.
Since the $B( \Phi[\bi])$ with  $\bi \in \Lambda_\rho$ have diameter at least $\rho$ and $x_2-x_1 \leq \rho$, the conclusion of the lemma follows.
\end{proof}

\medskip

 We now write
\begin{equation*}
N^\omega(x,\rho) = \# \{\bi \in \Lambda_\rho^\omega : B(\Phi[\bi])  \cap L_{x} \neq \emptyset\}
\end{equation*}
for the random analogue of (\ref{nxr}).
Fix $\epsilon \in (0,\alpha)$ and for $\rho\in(0,c_0)$ let $S_\rho$ be the deterministic subset of $V$:
\begin{equation}\label{sr}
S_\rho = \{x \in V:N(x,\rho) \leq \rho^{-\alpha +\epsilon/2}\}.
\end{equation}

The second subsidiary lemma shows that  if $x \in S_\rho$ then the probability that $L_x$ has non-empty intersection with $\Phi(\Sigma^\omega)$ is small.

\begin{lemma}\label{sub2}
Let  $\rho\in(0,c_0)$ and let $I \subset V$ be an interval with $|I| \leq \rho$ such that $I \cap S_\rho \neq \emptyset$. Then 
\begin{equation}\label{sub2eqn}
\mathbb{P}\big(N^\omega( x,\rho) >0 \mbox{\rm{ for some }} x \in  I \cap S_\rho \big) \leq 2 c_2 \rho^{\epsilon/2}.
\end{equation}
\end{lemma}

\begin{proof} 
If  $I \cap S_\rho= \emptyset$ then (\ref{sub2eqn}) is trivial. Otherwise,  applying Lemma \ref{sub1} to the interval $I$, taking $M = \rho^{-\alpha +\epsilon/2}$ and noting (\ref{sr}), we may find $x_1, x_2 \in I \cap S_\rho$ such that,  for all $x \in I \cap S_\rho$, all $\omega$,  and all $\bi \in \Lambda_\rho^\omega\subset  \Lambda_\rho$ with $B(\Phi[\bi]) \cap L_{x} \neq \emptyset$, either $B(\Phi[\bi]) \cap L_{x_1}\neq \emptyset$ or $B(\Phi[\bi]) \cap L_{x_2} \neq \emptyset$. In particular, for all $x \in I \cap S_\rho$
\begin{equation}\label{npx}
N^\omega(x,\rho) \leq N^\omega(x_1,\rho) + N^\omega(x_2,\rho).
\end{equation}
For $j=1,2$, using (\ref{prob}) and (\ref{sr}),
\begin{eqnarray*}
\mathbb{E} \big(N^\omega(x_j,\rho)\big) &=& 
\sum\big\{\mathbb{P}(\bi \in \Lambda_\rho^\omega):  \bi \in \Lambda_\rho  , B(\Phi[\bi]) \cap L_{x_j} \neq \emptyset\big\}\\
 &\leq&
 \sum\big\{c_2 \rho^{\alpha}: \bi \in \Lambda_\rho, B(\Phi[\bi]) \cap L_{x_j} \neq \emptyset\big\}\\
&\leq& c_2 \rho^{\alpha}N(x_j,\rho)\\
&\leq&  c_2 \rho^{\alpha}  \rho^{-\alpha +\epsilon/2},
\end{eqnarray*}
 so
$$\mathbb{P}\big(N^\omega(x_j,\rho) >0\big) \leq c_2  \rho^{\epsilon/2}.$$
The conclusion (\ref{sub2eqn}) follows  from (\ref{npx}). 
\end{proof} 
\medskip

Let
\begin{equation*}
 S = \{x \in V:\lbd ( \Phi(\Sigma)\cap L_x) \leq \alpha-\epsilon\}.
\end{equation*}
Note that, for all $\rho\in(0,c_0)$, we have 
$$\Phi(\Sigma) \cap L_x \subset \bigcup_{\bi \in \Lambda_\rho} \Phi[\bi]\cap L_x.$$
Thus, from (\ref{sr}),  (\ref{nxr}) and  (\ref{lbd}),
$$ S \subset \bigcap_{N=N_0}^\infty \bigcup_{n=N}^\infty S_{2^{-n}},$$
where we choose $N_0$ so that  $0< 2^{-N_0}<c_0$. 

The final subsidiary lemma essentially shows that  the Hausdorff dimension of $S$ cannot be too big.
\begin{lemma}\label{sub3}
With $S$ as above, $\dim_H (\pi_V(\Phi(\Sigma^\omega)) \cap S) \leq 1-\epsilon/4$ almost surely.
\end{lemma}

\begin{proof} 
For $\rho \in (0,c_0)$ write 
$$K^\omega_\rho := \bigcup \{B(\Phi[\bi]): \bi \in \Lambda^\omega_\rho  \}  \supset \Phi(\Sigma^\omega).$$
Let $I \subset V$ be an interval with $ |I| = \rho \leq c_0$. If $S_\rho \cap I \neq \emptyset$ then by Lemma \ref{sub2}
$$\mathbb{P}\big(\pi_V(K^\omega_\rho ) \cap S_\rho \cap I \neq \emptyset \big)\leq 2c_2 \rho^{\epsilon/2}.$$  
For $n\geq N_0$, let ${\mathcal C}_n$ be the family of closed binary subintervals of $V$ of lengths $2^{-n}$. Thus, for $n\geq N_0$,
$$\mathbb{E}\big(\# j : \pi_V(K^\omega_{2^{-n}} )\cap S_{2^{-n}} \cap I_j \neq \emptyset, I_j \in {\mathcal C}_n \big)
\leq 2^{n+1}|\Phi(\Sigma)|2 c_2 2^{-n\epsilon/2} = c_3 2^{n(1-\epsilon/2)}.$$ 
In particular,
\[
\sum_{n=N_0}^\infty 2^{-n(1-\epsilon/4)}\mathbb{E}\big(\# j : \pi_V(K^\omega_{2^{-n}} ) \cap S_{2^{-n}}\cap I_j \neq \emptyset, I_j \in {\mathcal C}_n \big)= c_3\sum_{n=N_0}^\infty 2^{-\epsilon/4} <\infty.
\]
Then, for all $N \geq N_0$,
\begin{eqnarray*}
\pi_V(\Phi(\Sigma^\omega)) \cap S & \subset & \pi_V(\Phi(\Sigma^\omega)) \cap \bigcup_{n=N}^\infty S_{2^{-n}}\\
& = &  \bigcup_{n=N}^\infty \pi_V(\Phi(\Sigma^\omega)) \cap S_{2^{-n}}\\
& \subset &  \bigcup_{n=N}^\infty \pi_V(K^\omega_{2^{-n}}) \cap S_{2^{-n}}\\
& \subset &  \bigcup_{n=N}^\infty \bigcup_{I_j \in {\mathcal C}_n}
\big\{I_ j : \pi_V(K^\omega_{2^{-n}} ) \cap S_{2^{-n}}\cap I_j \neq \emptyset\big\}.
\end{eqnarray*} 
Hence, writing ${\mathcal H}^s_\delta$ for the $s$-dimensional Hausdorff $\delta$-premeasure, and 
${\mathcal H}^s$ for $s$-dimensional Hausdorff measure, it follows on taking these covers of
 $\pi_V(\Phi(\Sigma^\omega)) \cap S$ for each $N$ that
\begin{align*}
\mathbb{E}\big( {\mathcal H}^{1-\epsilon/4}(\pi_V&(\Phi(\Sigma^\omega)) \cap S)\big)\\
&= 
\mathbb{E}\big(\lim_{N \to\infty} {\mathcal H}^{1-\epsilon/4}_{2^{-N}}(\pi_V(\Phi(\Sigma^\omega)) \cap S)\big)\\
&\leq 
\limsup_{N \to\infty}\mathbb{E}\big( {\mathcal H}^{1-\epsilon/4}_{2^{-N}}(\pi_V(\Phi(\Sigma^\omega)) \cap S)\big)\\
&\leq
  \mathbb{E}\Big( \sum_{n=N_0}^\infty 2^{-n(1-\epsilon/4)}\big(\# j : \pi_V(K^\omega_{2^{-n}} ) \cap S_{2^{-n}}\cap I_j \neq \emptyset, I_j \in {\mathcal C}_n\big) \Big) \\
 &< \infty.
\end{align*}
It follows that almost surely ${\mathcal H}^{1-\epsilon/4}(\pi_V(\Phi(\Sigma^\omega)) \cap S)< \infty$ and so 
$\dim_H (\pi_V(\Phi(\Sigma^\omega)) \cap S) \leq 1-\epsilon/4$.
\end{proof} 
\medskip

To complete the proof of Proposition \ref{prop2}, note that 
$$\dim_H \pi_V(\Phi(\Sigma^\omega)) = \max\big\{  \dim_H (\pi_V(\Phi(\Sigma^\omega)) \cap S),\,
\dim_H (\pi_V(\Phi(\Sigma^\omega)) \setminus S)\big\}$$
so that, conditional on $\dim_H \pi_V(\Phi(\Sigma^\omega))=1$,  an event of positive probability by the hypothesis of the proposition,
 $$1 \leq \max\big\{ 1-\epsilon/4, \dim_H (\pi_V(\Phi(\Sigma^\omega)) \setminus S) \big\}
 \leq \max\big\{ 1-\epsilon/4, \dim_H (\pi_V(\Phi(\Sigma)) \setminus S) \big\}.$$
But this is a deterministic statement, so we conclude that $ \dim_H (\pi_V(\Phi(\Sigma)) \setminus S)=1$. 
\qed
\medskip

\section{Sections of self-similar sets and percolation}
\setcounter{theorem}{0}

Next we obtain a weak dimension conservation property for the lower box-counting dimension of sections for self-similar sets with dense rotations (Theorem \ref{dimconthm}). We also do so for the Hausdorff dimension of sections of Mandelbrot percolation sets (Theorem \ref{mandper}). 

The best known model of fractal percolation is Mandelbrot percolation, based on a decomposition of the $d$-dimensional cube into $M^d$ equal subcubes of sides $M^{-1}$; its topological properties  have been studied  extensively, see \cite{Dek, Falconer14,RS3}. Nevertheless, statistically self-similar subsets of any self-similar set may be constructed using a similar percolation process which may be set up in terms of the symbolic space formulation of Section \ref{sec2}.

Let $ \mathcal{I}= \{f_1,\ldots,f_m\}$ be an IFS of similarities with attractor  $K$. Intuitively, percolation on $K$ is performed by retaining or deleting components of the natural hierarchical construction of $K$ in a  self-similar random manner. Starting with some non-empty compact set $D$ such that $f_i(D) \subset D$ for all $i$, we select a subfamily of the sets $\{f_1(D), \ldots, f_m(D)\}$ according to  some probability distribution, and write $K^1$ for the union of the selected sets. Then, for each  selected $f_i(D)$, we choose subsets from
$\{f_{i}f_{1}(D), \ldots, f_{i}f_{m}(D)\}$ according to the same probability distribution, independently for each $i$, with the union of these sets comprising $K^2$. Continuing in this way, we get a nested hierarchy $K \supset K^1\supset K^2 \supset \cdots$ of random compact sets, where $K^k$ denotes the union of the components remaining at the $k$th stage. The random percolation set $K^\omega \subset K$ is then given by 
$K^\omega = \cap_{k=0}^\infty K^k$, see Figure 1.

\begin{figure}[h]\label{fig1}
\begin{center}
\includegraphics[scale=0.2]{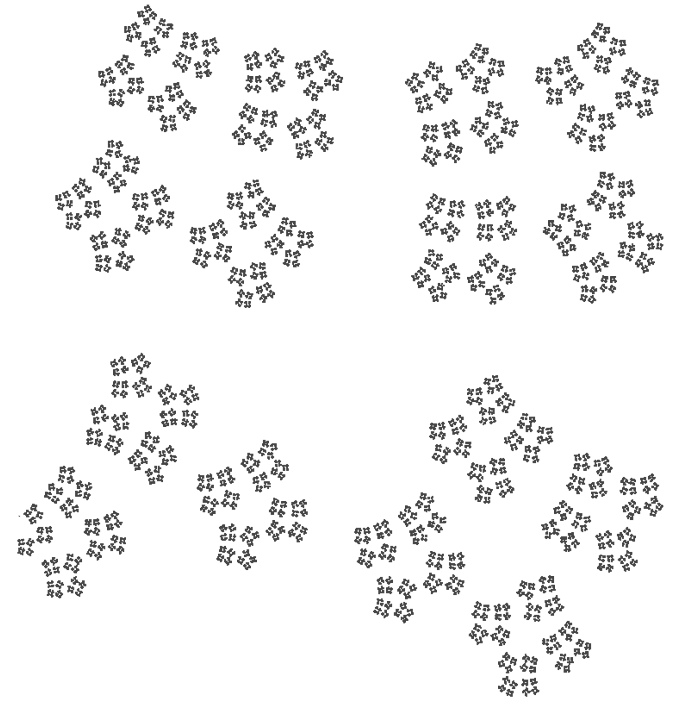}
\qquad\qquad
\includegraphics[scale=0.2]{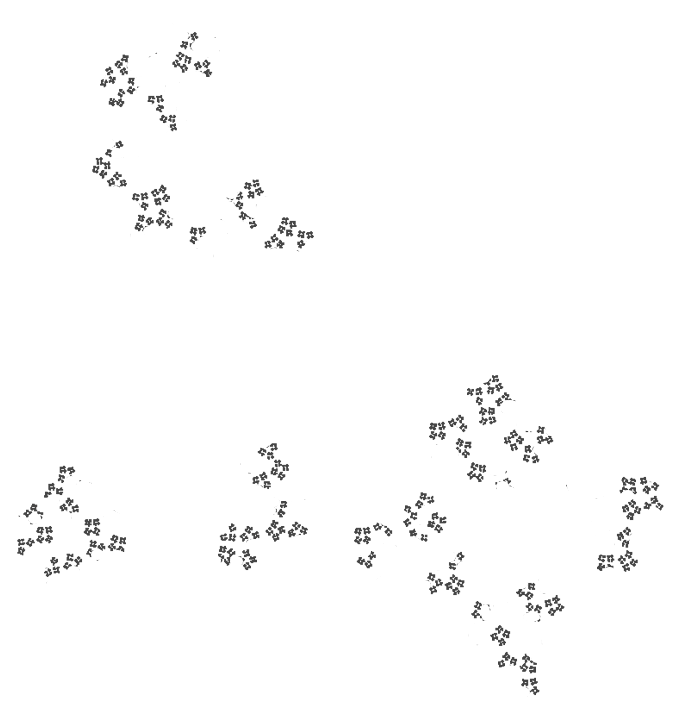}
\caption{A self-similar attractor of an IFS with rotations and a subset obtained by the percolation process}
\end{center}
\end{figure}

More formally, percolation on a self-similar set $K$ is defined using the natural representation of $K$ by symbolic space.  As in Section \ref{sec2} we take $ \Lambda =\{1,\ldots,m\}$ with 
$\Sigma_* = \cup_{n \geq 0} \Lambda^n$ the set of finite words and   $\Sigma = \Lambda^\mathbb{N}$ the  infinite words. 
The canonical map $\Phi:\Sigma \to K \subset \mathbb{R}^d$  is given by $\Phi (i_1i_2\ldots) =  \cap_{n=0}^\infty f_{i_1} \cdots  f_{i_n}(D)$ for any non-empty compact set $D$ such that $f_i(D) \subset D$ for  $ i= 1,\ldots,m$. Then $K =  \cup_{\bi \in \Sigma}\Phi(\bi)$, with $\Phi$ providing a (not necessarily injective) index to  the points of $K$.

To define percolation on $K$, let $(\Omega,\mathcal{A},\mathbb{P})$ be a probability space. Let $X \equiv (X_1, \ldots, X_m)$ be a random vector taking values in $\{0,1\}^m$. Let ${\mathcal X} = \{X^\bi \equiv (X^\bi_1, \ldots, X^\bi_m) \}_{\bi \in \Sigma_*}$ be a family of independent random vectors with values in  $\{0,1\}^m$, each having the distribution of $X$, on the probability space $(\Omega^{\Sigma_*},\mathcal{A}_\mathcal{X}, \, \mathbb{P}^{\otimes \Sigma_*})$, where $\mathcal{A}_\mathcal{X}\subset \mathcal{A}^{\Sigma_*}$ is the $\sigma$-algebra generated by $\mathcal{X}$.
This defines a random set $\Sigma^\omega = \{i_1i_2\ldots \in \Sigma : 
X^{\emptyset}_{i_1}X^{i_1}_{i_2} X^{i_1i_2}_{i_3}\cdots = 1\} \subset \Sigma$.
The {\it percolation set} $K^\omega\subset K$ is the image of $\Sigma^\omega$ under the canonical map, that is the random set $K^\omega = \Phi (\Sigma^\omega)$.

By standard branching process theory \cite{AN}, if $\mathbb{E}(\#i : X_i =1)   > 1$ there is a positive probability that  $\Sigma^\omega$, and thus $K^\omega$, is non-empty. Provided the IFS defining $K$ satisfies OSC then, conditional on  $K^\omega\neq \emptyset$,
\begin{equation}\label{percdim}
\dim_B K^\omega=\dim_H K^\omega = s \text{ a.s where } s \text{ satisfies }
\mathbb{E}\Big(\sum_{i=1}^m  X_i r_i^s\Big) = 1,
\end{equation}
where $r_i$  is the scaling ratio of $f_i$, see \cite{Fa86,MW86}.

We say that the percolation process is {\it standard with exponent} $\alpha$ if the distribution  
of $X = (X_1, \ldots, X_m)$ is defined by 
$\mathbb{P}(X_i= 1) = r_i^\alpha, \,  \mathbb{P}(X_i= 0) = 1- r_i^\alpha$
 independently for $i=1,\ldots,m$. Then by (\ref{percdim}), provided that $\alpha <\dim_H K$, there is a positive probability that $K^\omega \not= \emptyset$, in which case 
$\dim_H K^\omega = \dim_H K - \alpha$ a.s..

The following theorem on the dimension of projections of percolation subsets of self-similar sets was obtained as a corollary of a more general theorem on projections of random cascade measures on self-similar sets \cite{FJ}.

\begin{theorem}{\cite{FJ}}\label{properc}
Let  $K$ be the attractor of an IFS of contracting similarities on $\mathbb{R}^d$  with dense rotations and satisfying OSC. Let $\mathbb{P}$ be a probability distribution  of a standard percolation process on $K$ with 
$\mathbb{E} (\#i : X_i =1)   > 1$, so that the percolation set $K^\omega \not= \emptyset$ with positive probability.  Then, conditional on $K^\omega \not= \emptyset$, almost surely  
\[
\dim_H \pi_V (K^\omega)=\min(k,\dim_H K^\omega),
 \]
for every $k$-dimensional subspace $V$.
\end{theorem}
Thus, conditional on non-extinction, the projections of $K^\omega$ onto  {\it all} subspaces have the `generic' dimension.
We now apply Proposition \ref{prop2} to  sections of self-similar sets. The conclusion applies to self-similar sets $K$ such that their projection onto each line is the same as that of the convex hull of $K$. This includes the case where $K$ is connected as well as many other self-similar sets, see Figure 2.

\begin{figure}[h]\label{fig2}
\begin{center}
\includegraphics[scale=0.23]{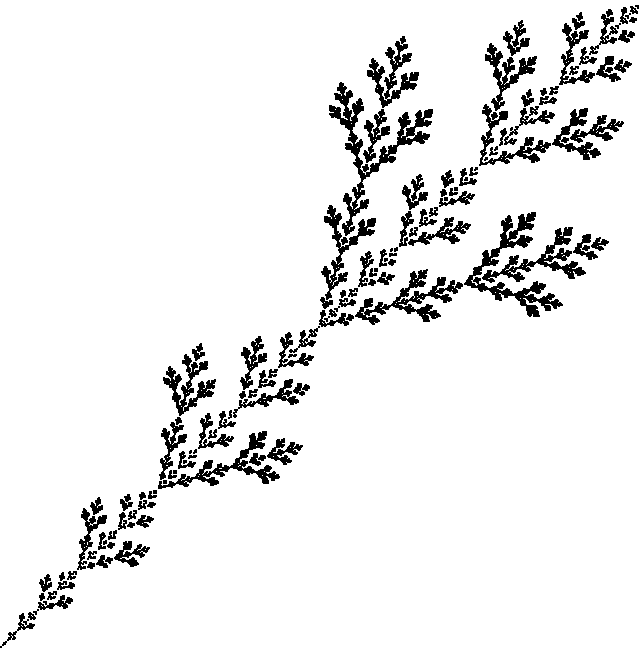}
\quad
\includegraphics[scale=0.185]{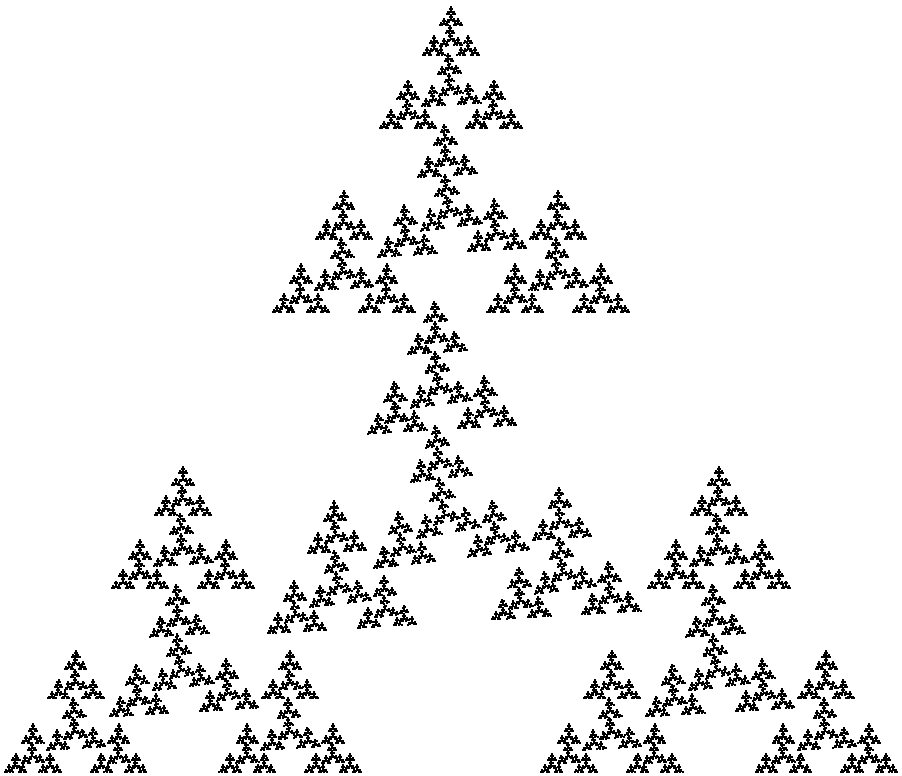}
\caption{A connected and a totally disconnected self-similar set with dense rotations satisfying the conditions of Theorem \ref{dimconthm}}
\end{center}
\end{figure}

\begin{theorem}\label{dimconthm}
Let $\mathcal{I}$  be an IFS of contracting similarities on $\mathbb{R}^d$  with dense rotations and satisfying OSC. Let $K$ be the attractor of $\mathcal{I}$ and suppose $s= \dim_H K >1$ and that the projection of $K$ onto every 1-dimensional subspace equals that of its convex hull.
Then for every 1-dimensional subspace $V$ of $\mathbb{R}^d$ and all $\epsilon\in(0,s-1)$, 
\begin{equation}\label{conclu}
\dim_H \big\{x\in V:\lbd\big(K\cap \pi_V^{-1}(x)\big)> \dim_H K - 1 -\epsilon\big\}=1.
\end{equation}
\end{theorem}

\noindent{\it Proof.}
Let $K$ have its symbolic representation $\Phi: \Sigma \to \mathbb{R}^d$.   As $\Phi[\bi]$ is similar to $K$ for all $\bi \in \Sigma_*$, the projection of each $\Phi[\bi]$ onto every 1-dimensional subspace is the same as that of its convex hull. We set up standard percolation with exponent $ s-1$ on $K$ via its symbolic representation, as above. Then there is a positive probability of non-extinction, conditional on which almost surely, 
$\dim_H \pi_V (K^\omega)=\min\{1,\dim_H K - (s-1)\}=1$ for  every line $V$, using Theorem \ref{properc}.  

A consequence of OSC is that $\Phi$ satisfies conditions (1) and (2) (at (\ref{lambdar}) and (\ref{overlap})) with $c_0 = |K|$ and $c_1 = \max_{1 \leq i \leq m}r_i^{-1}$. Moreover, if $i_1\ldots i_k \in \Lambda_\rho$ then
$\mathbb{P} (i_1\ldots i_k \in \Lambda^\omega_\rho) = r_{i_1}^\alpha \cdots r_{i_k}^\alpha 
\leq |K|^{-\alpha} \rho^\alpha$, so that (\ref{prob}) is satisfied. The conclusion follows by Proposition \ref{prop2} since $\Phi(\Sigma) = K$. 
\hfill$\Box$
\medskip

It would be desirable to dispense with the requirement in Theorem \ref{dimconthm} that the projections of $K$ are the same as those of its convex hull. Without such a condition it is not hard  to show that \eqref{conclu} can be replaced by the conclusion that 
$$\dim_H \big\{x\in V:d(x)> \dim_H K - 1 -\epsilon\big\}=1$$
where 
$d(x):=\varliminf_{\rho\to 0} \log \# N_\rho(L_x^\rho)/-\log\rho$ and where $N_\rho(L_x^\rho)$ denotes the number of $\bi \in \Lambda_\rho$ such that $B(\Phi[\bi]) \cap L_y \neq \emptyset$ for some $y \in [x-\rho,x+\rho]$. (Here $d(x)$ is a kind of lower box-counting dimension conditioning on fibres that  is always no less than the actual lower box-counting dimension of the fibre, with possibility of being strictly larger.)

Next we apply Proposition \ref{prop1} to Mandelbrot percolation. Let $K$ be the unit cube in $\mathbb{R}^d$. 
Fix an integer $M \geq 2$ and  a probability $0<p<1$. We divide $K$ into $M^d$ subcubes of side $1/M$ in the natural way, and retain each subcube independently with probability $p$ to get a set $K^1$ formed as a union of the retained subcubes. We repeat this process with the cubes in $K^1$, dividing each into $M^d$ subcubes of side $1/M^2$ and choosing each with probability $p$ to get a set $K^2$, and so on. This process, termed {\it Mandelbrot percolation}, leads to a percolation set, which we write here as $K^\omega_p = \cap_{k=0}^\infty K^k$ to emphasize the dependence on $p$.

Of course, this may be regarded as  percolation on the self-similar set defined by the IFS 
${\mathcal I} = \{f^{j_1,\ldots,j_d} : 1 \leq j_1,\ldots,j_d\leq M\}$ on $\mathbb{R}^d$ where 
$$f^{j_1,\ldots,j_d}(x_1, \ldots,x_d) = \Big(\frac{x_1+ j_1-1}{M}, \ldots,\frac{x_d+ j_d-1}{M}\Big);$$ 
as before the random construction may be represented in symbolic space, using an alphabet of $M^d$ letters.

If $p > M^{-d}$ then, as above, that there is a positive probability that $K^\omega_p\neq \emptyset$, conditional on which
$\dim_H K^\omega_p = d+\log p / \log M$. A useful observation  is that for $0<p,p' <1$  the intersection of independent realizations of the two random sets $K^\omega_{p} $ and $K^\omega_{p'} $ has the same distribution as that of $K^\omega_{p p'}$.

Rams and Simon \cite{RS,RS2, RS3} and Simon and V\'{a}g\'{o} \cite{SV} have recently obtained results on the dimensions and Lebesgue measure of projections of Mandelbrot percolation that are almost surely valid for projections onto all subspaces. 

\begin{theorem}{\cite{RS,SV}}\label{ramssimon}
Let $1 \leq k \leq d-1$ and let $K^\omega_p \subset \mathbb{R}^d$ be the random set obtained by Mandelbrot percolation on the $d$-dimensional unit cube, using repeated subdivision into $M^d$ subcubes, and selecting cubes independently with probability $p> 1/M^{d-k}$. Then, conditional on  $K^\omega_p \neq \emptyset$, $\dim_H K^\omega_p = d+\log p / \log M>k$, and for every $k$-dimensional subspace $V$ we have
${\mathcal L}^k(\pi_V K^{\omega}_{p})>0$, indeed, 
$\pi_V K^\omega_{p}$ contains an open subset of $V$. 
\end{theorem}

Applying Proposition \ref{prop1} to Theorem \ref{ramssimon} we obtain dimension conservation properties for Mandelbrot percolation.

\begin{theorem}\label{mandper}
Let $1 \leq k \leq d-1$. Let $K^\omega_p \subset \mathbb{R}^d$ be the random set obtained by Mandelbrot percolation on the $d$-dimensional unit cube, using repeated subdivision into $M^d$ subcubes and selecting cubes independently with probability $p> 1/M^{d-k}$. For all $\epsilon>0$, almost surely conditional on  $K^\omega_p \neq \emptyset$, for all $k$-dimensional subspaces $V$,
 \begin{equation*}
\mathcal{L}^k\big\{x\in V: \dim_H \big(K^\omega_p \cap \pi_V^{-1}(x) \big)\geq \dim_H K^\omega_p - k-\epsilon\big\}>0.\end{equation*}
\end{theorem}

\noindent{\it Proof.}
We may represent the heierarchy of $M$-ary subcubes of the unit cube  in symbolic space $\Sigma$ with  an alphabet  $\Lambda$ of $m=M^d$ letters  with $\Phi: \Lambda \to K = [0,1]^d$ the natural cannonical mapping. With notation for percolation as above, let the probability distribution $(X_1, \ldots, X_m)$ on $\Lambda$ be given by $\mathbb{P}(X_i= 1) = p, \,  \mathbb{P}(X_i= 0) = 1-p$, independently for $i=1,\ldots,m$. This defines a random set $\Sigma^\omega_p \subset \Sigma$ such that $K^\omega_p =\Phi(\Sigma^\omega_p)$ is the Mandelbrot percolation set, with $\dim_H K^\omega_p =  d+ \log p/ \log M$ conditional on non-extinction. 
Now let $p' = p^{-1}M^{-(d-k-\epsilon)}$ and let $\Sigma^{\omega'}_{p'} \subset \Sigma$ be an independent random set defined in the same way but using probability $p'$; we use $\Sigma^{\omega'}_{p'}$ to `probe' the dimensions of $K^\omega_p$.

The random set  $\Sigma^\omega_p \cap \Sigma^{\omega'}_{p'}$ has the same distribution as  a random set $\Sigma^{\omega''}_{pp'}$, constructed  in the same way with probability $pp'$.
Thus, conditional on $\Sigma^\omega_p \cap \Sigma^{\omega'}_{p'}\neq \emptyset$,
 $\dim_H \Phi(\Sigma^\omega_p \cap \Sigma^{\omega'}_{p'})=  d+ \log pp'/ \log M = k  +\epsilon$ almost surely, so by Theorem \ref{ramssimon}, almost surely,
  \begin{equation}\label{posmes}
{\mathcal L}^k \big(\pi_V (\Phi (\Sigma^\omega_p \cap \Sigma^{\omega'}_{p'})\big)>0
 \end{equation}
for all $k$-dimensional subspaces $V$. 
Using independence and Fubini's theorem,  conditional on  $\Sigma^{\omega}_{p} \neq \emptyset$, almost surely conditional on $\Sigma^\omega_p \cap \Sigma^{\omega'}_{p'} \neq \emptyset$, inequality (\ref{posmes}) holds for all $V$
(Note that, conditional on $\Sigma^{\omega}_{p} \neq \emptyset$, 
$ \P(\Sigma^\omega_p \cap \Sigma^{\omega'}_{p'} \neq \emptyset)>0$.)
 
We may regard $\Sigma^{\omega'}_{p'}$ as an $\alpha$-random subset of $\Sigma$ where 
$\alpha = -\log p'/\log M  = \log p/\log M +d-k -\epsilon$.
Taking $A=\Sigma^\omega_p$ in   Proposition \ref{prop1} (so in the notation there $A^\omega =\Sigma^\omega_p \cap \Sigma^{\omega'}_{p'}$) we conclude that, conditional on  $\Sigma^{\omega}_{p} \neq \emptyset$, 
$$\mathcal{L}^k\big\{x\in V: \dim_H \big(\Phi(\Sigma^\omega_p) \cap \pi_V^{-1}(x) \big)\geq \alpha\}>0,$$
and  the conclusion follows, noting that $\Phi(\Sigma^\omega_p) = K^\omega_p$.
\qed
\medskip

\section{Absolute continuity of projections of random self-similar measures}
\setcounter{theorem}{0}
We now show that we have weak dimension conservation for the Hausdorff dimension of sections of plane self-similar sets in all directions apart from a set of directions of Hausdorff dimension $0$ (Theorem \ref{thmhdgen}).
To achieve this we use Proposition \ref{prop1} together with a result on the absolute continuity of projections of a class of random measures supported by random subsets of self-similar sets (Theorem \ref{absss}), which is a  extension of a  result  of Shmerkin and Solomyak \cite{SS} for deterministic measures. We do this first for self-similar sets where the defining similarities are translates of each other. Then a device of Peres and Shmerkin \cite{PS} enables us to extend the conclusion to general similarities.
 
Let 
\begin{equation}\label{mes1}
\mathcal{I}=\{f_i=rR_\theta\cdot +a_i\}_{i=1}^m
\end{equation}
 be an IFS in the plane, where $r\in(0,1)$ and $R_\theta$ is the orthogonal rotation with an angle $\theta\in [0,2\pi)$. As before $\Phi:\Sigma \mapsto \mathbb{R}^2$ is the canonical mapping from the symbolic space to the plane.

Let $(\Omega,\mathcal{A},\mathbb{P})$ be a probability space. Let
\begin{equation}\label{ifsr}
X:\Omega\mapsto \Big\{(p_1,\ldots,p_m)\in[0,1]^m: \sum_{i=1}^m p_i=1\Big\}
\end{equation}
be a random probability vector allowing zero entries. For $n\in \mathbb{N}$ denote by
\[
\chi_n: \Omega^\mathbb{N}\to \Omega
\]
the projection from $\Omega^\mathbb{N}$ onto its $n$th coordinate. Then $\mathcal{X}=\{X^{(n)}=X\circ \chi_n\}_{n\in \mathbb{N}}$ forms a i.i.d. sequence on the probability space $(\Omega^\mathbb{N},\mathcal{A}_\mathcal{X}, \mathbb{P}^{\otimes \mathbb{N}})$, where $\mathcal{A}_\mathcal{X}\subset \mathcal{A}^{\otimes \mathbb{N}}$ is the $\sigma$-algebra generated by $\mathcal{X}$. Let $\nu$ be the random probability measure on $\Sigma$ defined  by
\begin{equation}\label{mes2}
\nu([i_1\ldots i_k])=X^{(1)}_{i_1}\cdots X^{(k)}_{i_k} \ \mbox{ for all }\ i_1\ldots i_k \in \Sigma_*.
\end{equation}

\begin{remark}
Note that the measure $\nu$ is not the same as the random cascade measures studied, for example, in \cite{FJ}. Here for $k\ge 1$ the ratio 
$\nu([i_1\ldots i_k i_{k+1}]):\nu([i_1\ldots i_k])$ is the same for all $i_1\ldots i_k \in \Lambda^k$.
The reason why we consider this particular random measure is that its Fourier transform has a convolution structure, which is essential for the proof of absolute continuity in Theorem \ref{absss}.
\end{remark}

Let $\mathbb{Q}$ be the probability measure on the product space $\Sigma\times \Omega^\mathbb{N}$ given by
\[
\mathbb{Q}(A)=\int_{\Omega^\mathbb{N}}\int_\Sigma \mathbf{1}_A(\bi,\boldsymbol{\omega}) \, \nu(\mathrm{d}\bi)\, \mathbb{P}^{\otimes\mathbb{N}}(\mathrm{d}\boldsymbol{\omega}) \ \mbox{ for all } A\in \mathcal{B}_\Sigma\otimes\mathcal{A}_\mathcal{X}.
\]
Denote by $\sigma:\Sigma\times \Omega^\mathbb{N}\mapsto \Sigma\times \Omega^\mathbb{N}$ the left shift \[
\sigma(i_1i_2\ldots, \omega_1\omega_2\ldots)=(i_2i_3\ldots, \omega_2\omega_3\ldots).
\]

The next proposition and theorem are direct analogues of those obtained in \cite{FJ} for random cascade measures.

\begin{proposition}\label{mixing}
The dynamical system $(\Sigma\times \Omega^\mathbb{N},\mathcal{B}_\Sigma\otimes\mathcal{A}_\mathcal{X},\sigma, \mathbb{Q})$ is mixing.
\end{proposition}

\begin{proof}
The proof is similar to that of  \cite[Proposition 2.2]{FJ}.  Let $\mathcal{B}$ be the semialgebra of $\mathcal{B}_\Sigma\otimes\mathcal{A}_\mathcal{X}$ consisting of sets of the form
\[
\{(\bi,\boldsymbol{\omega}): \bi |_k={\bf j},X^{(b)}_a\in B_a^b\}
\]
for $k\in \mathbb{N}$, ${\bf j}\in \Lambda^k$, $b\in \{1,\ldots,k\}$, $a\in \Lambda$ and $B^b_a$ Borel subsets of $[0,1]$.  It is clear that $\mathcal{B}$ generates $\mathcal{B}_\Sigma\otimes\mathcal{A}_\mathcal{X}$, so we only need to verify that $\lim_{n\to\infty}\mathbb{Q}(\sigma^{-n}(A)\cap B)=\mathbb{Q}(A)\mathbb{Q}(B)$ for $A,B\in \mathcal{B}$. This follows since  by the construction of $\mathcal{B}$, given $A,B\in\mathcal{B}$, there exists a positive integer $n_0$ such that $\sigma^{-n}(A)$ and $B$ are independent for all $n\ge n_0$. 
\end{proof}

Let $\pi_\beta:\mathbb{R}^2\mapsto \mathbb{R}^2$ be orthogonal projection onto the line making an angle $\beta$ with the $x$-axis. Write $\mu=\Phi\nu$ for the measure defined by $\mu(A) = \nu(\Phi^{-1}A)$. Starting from Proposition \ref{mixing} and proceeding just as in \cite{FJ}, we obtain the following projection property.

\begin{theorem}\label{proj}
Suppose that $\theta/\pi$ is irrational. Then almost surely, for all $\beta\in [0,\pi)$
\[
\dim_H \pi_\beta \mu=\min(1,\dim_H \mu).
\]
\end{theorem}

\begin{proof}
When $\theta/\pi$ is irrational, the closed rotation group $G$ generated by $R_\theta$ is the whole group $SO(2,\mathbb{R})$. Given this, the proof follows exactly the same lines as in  of \cite[Sections 2.7 \& 4]{FJ}. In particular, since $G=SO(2,\mathbb{R})$, the dimension of the projections equals the maximal possible value, just as in   \cite[Corollary 4.6]{FJ},
\end{proof}

Theorem \ref{absss} below, a random analogue of  \cite[Theorem B]{SS},  gives conditions  for the projections of the random measure $\mu$ to be almost surely absolutely continuous in all directions except for a set $E$ of Hausdorff dimension $0$. First, in the following lemma, we specify the set $E$ and verify that its dimension is $0$. We adapt the delicate estimates of \cite[Lemmas 3.2 \& 3.4]{SS}  to our requirements,  in particular obtaining  estimates for the dimensions of $E_{q,k}(\delta,N)$ that do not depend on $q$ or $k$.

For $x\in\mathbb{R}$ let $\|x\|=\min\{|x-j|:j\in\mathbb{Z}\}$ and we write  $[N]=\{1,\ldots,N\}$ for each positive integer $N$. 

\begin{lemma}\label{cover}
Fix $r\in(0,1)$,  $\gamma\in\mathbb{R}$, $b\in(0,\infty)$ and $\theta\in \mathbb{R}$ with $\theta/\pi$ irrational. 
For $\delta\in(0,\frac{1}{2})$ and integers $q,k\ge 1$, $N\ge 2$, let $E_{q,k}(\delta,N)$ be the set of all  $\beta\in[0,\pi)$ such that
\[
 \max_{\tau\in[1,r^{-qk}]}\frac{1}{N}\#\Big\{n\in[N]:\|b\tau r^{q-qk(N-n)}\cos(\beta + \gamma - nqk\theta)\|\le \frac{r^{2qk}}{15} \Big\}>1-\delta,
 \]
and let
\[
E= \bigcap_{i\ge 3} \bigcup_{q,k\ge 1}\limsup_{N\to\infty}E_{q,k}(1/i,N).
\]
Then $\dim_H E = 0$.
 \end{lemma}

\begin{proof}
For the time being we fix the integers $q,k,N\ge 1$ and abbreviate $c:=br^{q}$, $\ell:=r^{-qk}$ and $\alpha:=qk\theta$. Note that $r^{2qk}/15=1/(15\ell^2)$. Let $\tau \in [1,\ell]$. 

Given $\beta \in [0,\pi)$, for each $n=1,\ldots,N$ write
\begin{equation}\label{knen}
c\tau \ell^{N-n}\cos(\beta + \gamma - n\alpha)=k_n+\epsilon_n, \text{ where } k_n\in \mathbb{Z}\text{ and } \epsilon_n\in[-1/2,1/2).
\end{equation}
For $x\in\mathbb{R}$ let $w_x=(\cos x,\sin x)$. Since $\alpha/\pi$ is irrational, the unique solution of the equation
\[
c_{1}w_{2\alpha}+c_{2}w_{\alpha}
=w_{0},
\]
 is $c_{1}=-1$ and $c_{2}=2\cos \alpha$. Clearly $|c_{1}|, |c_{2}|\le 2$.

Applying the formula $\langle w_{x},w_{\beta+\gamma-n\alpha}\rangle=\cos(\beta+\gamma- n\alpha-x)$ for $x=2\alpha,\alpha,0$ and using \eqref{knen} we get that
\begin{equation}\label{kn12}
c_{1}\ell^2(k_{n+2}+\epsilon_{n+2})+c_{2}\ell(k_{n+1}+\epsilon_{n+1})=k_n+\epsilon_n.
\end{equation}
This implies that if
\[
\max\{|\epsilon_n|,|\epsilon_{n+1}|,|\epsilon_{n+2}|\}\le 1/(15 \ell^{2}) \le 1/ \big(3(2\ell^2+2\ell+1)\big),
\]
then
\[
|c_{1}\ell^2k_{n+2}+c_{2}\ell k_{n+1}-k_n|<\textstyle{\frac{1}{2}},
\]
which means that $k_{n+2}$ and $k_{n+1}$ uniquely determine $k_{n}$. On the other hand, 
\[
|c_{1}\ell^2\epsilon_{n+2}+c_{2}\ell\epsilon_{n+1}-\epsilon_n|\le\ell^2+\ell+1.
\]
Hence for fixed $k_{n+2}$ and $k_{n+1}$, there are at most $\lfloor 2(\ell^2+\ell+1)+1\rfloor \le 7\ell^2$  possible values of $k_n$. Also, from  \eqref{knen}, there are at most $(2c \ell+1)(2c \ell^2+1)\le (2b+1)^2 \ell^3$  possible pairs of $(k_N,k_{N-1})$. 

For $\delta\in(0,\frac{1}{2})$ denote by $[N]_\delta$ the set of all subsets of $[N]$ with cardinality  no less than $(1-\delta)N$. For $A\in[N]_\delta$ let $\widetilde{A}:=\{0\le n\le N-2:n+2,n+1,n\in A\}$. Then $\#\widetilde{A}\ge (1-3\delta)N-3$. This implies that the number of possible sequences $(k_n)_{n=0}^N$ corresponding to  $\beta\in[0,\pi)$ for which   $|\epsilon_n|\le 1/(15 \ell^{2}) $ in (\ref{knen}) for all $n\in A$, is bounded above by
\[
(2b+1)^2 \ell^3(7\ell^2)^{3\delta N +3}.
\]
Note that once $(k_N,k_{N-1})$ is given, the possible values of the remaining $k_n$ are determined by \eqref{kn12}, hence the value of $\tau\in [1,\ell]$ is irrelevant. Then, by Chernoff's entropy inequality for  binomial sums, see \cite{ES}, or alternatively using Stirling's approximation, 
\[
\# [N]_\delta \leq \sum_{p=0}^{\lfloor \delta N \rfloor} \binom{N}{p} 
\leq 2^{N[-\delta\log\delta -(1-\delta)\log(1-\delta)]}\leq \mathrm{e}^{C\sqrt{\delta}N},
\]
for all $N$ and $\delta \in (0,\frac{1}{2})$, where $C$ is a universal constant.

Combining these estimates, the number of possible sequences $(k_n)_{n=1}^N$ corresponding to $\beta\in [0,\pi)$ satisfying
\[
 \max_{\tau\in[1,\ell]}\frac{1}{N}\#\big\{n\in[N]:\|c \tau \ell^{N-n}\cos(\beta+\gamma-n\alpha)\|\le 
 1/(15\ell^{2})  \big\}>1-\delta,
\]
is bounded above by
\[
 \mathrm{e}^{C\sqrt{\delta}N} (2b+1)^2 \ell^3(7\ell^2)^{3\delta N+3}.
\]

From \eqref{knen}, identically 
\[
\beta+\gamma -n\alpha=\tan^{-1}\Big(\frac{\ell(k_{n+1}+\epsilon_{n+1})}{(k_n+\epsilon_n)\sin\alpha}-\cot\alpha\Big).
\]
Since $\alpha/\pi$ is irrational, by estimating the derivatives of the function
\[
f(x)=\tan^{-1}\big((\ell/\sin\alpha)\,  x-\cot\alpha\big),
\]
there is a constant $C'$ depending only on $\ell$ and $\alpha$ such that
\[
\beta\in B\big(j\alpha-\gamma+f(k_{j+1}/k_j), C'\ell^{-N}\big)
\]
where $j$ may be $1$ or $2$ (to ensure that $k_1$ and  $k_2$ are not both 0 when $N$ is sufficiently large). Hence the set $E_{q,k}(\delta,N)$
can be covered by 
$$2\mathrm{e}^{C\sqrt{\delta}N} (2b+1)^2 \ell^3(7\ell^2)^{3\delta N+3}=2\mathrm{e}^{C\sqrt{\delta}N} (2b+1)^2 r^{-3qk}(7r^{-2qk})^{3\delta N+3}$$
balls of radius $C'\ell^{-N}=C'r^{qkN}$.

Using these coverings, it follows that
\begin{eqnarray*}
\dim_H\big( \limsup_{N\to \infty} E_{q,k}(\delta,N)\big)&\leq& 
\frac{C\sqrt{\delta}+3\delta(\log 7-2qk\log r)}{-qk\log r}\\
&\leq& \big(6 +(C+3\log 7)/-\log r \big) \sqrt{\delta}.
\end{eqnarray*}
By countable stability of Hausdorff dimension,  for $i\ge 3$,
\[
\dim_H \bigcup_{q,k\ge 1}\limsup_{N\to\infty}E_{q,k}(1/i,N)\leq \big(6 +(C+3\log 7)/-\log r \big) /\sqrt{i},
\]
giving the conclusion.
\end{proof}

Here is the theorem on the absolute continuity of projections of random measures in all but a small set of exceptional directions  when the underlying similarities are translates of each other. The proof  uses Fourier transforms along  the lines of  \cite[Theorem B]{SS}.

\begin{theorem}\label{absss}
Suppose that $\theta/\pi$ is irrational and let $\mathcal{I}$ be an IFS of the form  \eqref{mes1} satisfying OSC. Then  there exists a set $E\subset [0,\pi)$ with $\dim_H E = 0$ such that, for every random self-similar measure  $\mu=\Phi\nu$ with respect to $\mathcal{I}$ of the form defined by \eqref{ifsr}--\eqref{mes2} and satisfying
\begin{equation}\label{uni}
\mathbb{P}(\mbox{\rm there exist } i,j \in\Lambda \  \mbox{\rm  such that } X_i, X_j\ge p_*)=1
\end{equation}
for some $p_*>0$ and
\begin{equation}\label{dim>1}
\mathbb{P}(\dim_H \mu=s)=1
\end{equation}
for some $s>1$, almost surely for all $\beta\in [0,\pi)\setminus E$, the projected measure $\pi_\beta \mu$ is absolutely continuous with respect to Lebesgue measure.
\end{theorem}

\begin{proof}
We write $\mbox{ang}(z)$ for the angle between the line containing $\{0,z\}$ and the $x$-axis. For $i,j\in \Lambda$ let $E_{i,j}$ be the set given by Lemma \ref{cover} for the ratio $r$ and angle $\theta$ in the IFS \eqref{mes1} with $\gamma=\mbox{ang} (a_{i}-a_{j})+\theta$ and $b=|a_{i}-a_{j}|$. Let $E=\cup_{i,j\in \Lambda} E_{i,j}$. Then $\dim_H E = 0$; we will show that the projected measures $\pi_\beta \mu$ are absolutely continous when $\beta \in [0,\pi)\setminus E$.

With $\mu=\Phi\nu$ as stated, we may, by \eqref{uni},  choose $i,j\in \Lambda$ with $|a_i-a_j|>0$ such that 
\begin{equation}\label{wij}
\mathbb{P}(X_{i}, X_{j}\ge p_*):=p>0;
\end{equation}
these $i$ and $j$ will remain fixed throughout the proof.

For each $q\ge 1$, we may regard the attractor  $K$ of the IFS \eqref{mes1} as the attractor of the iterated IFS
\[
{\mathcal I}_q := \{f_\bi:=f_{i_1}\cdots f_{i_q}\equiv r^qR_{q\theta}\cdot+a_\bi: \bi= i_1\ldots i_q \in \Lambda^q\},
\]
so that $K = \Phi_q( \Sigma_q)$ where
$\Sigma_q := \{\bi_1\bi_2\ldots: \bi_j \in \Lambda^q\}$ and $\Phi_q$ is the cannonical map. Let $\nu_q$ be the random self-similar measure of the form (\ref{ifsr})--(\ref{mes2}) with respect to
\[
\mathcal{X}_q=\Big\{X^{q,(n)}:=\Big(X^{q,(n)}_{\bi}\equiv \prod_{l=1}^qX_{i_l}^{(nq-q+l)}\Big)_{\bi=i_1\ldots i_q\in \Lambda^q}\Big\}_{n\ge 1}.
\]
Then  $\mu=\Phi_q\nu_q$ for all $q\ge 1$. Note that $\mu$ satisfies
\begin{equation}\label{sss}
\mu=\sum_{\bi \in \Lambda^q} X^{q,(1)}_\bi f_\bi\mu^{q,(1)},
\end{equation}
where $\mu^{q,(1)}$ is the copy of $\mu$ generated by $\{X^{q,(n+1)}\}_{n\in \mathbb{N}}$. In terms of Fourier transforms, writing $T_q=r^qR_{q\theta}$, equation \eqref{sss} yields that for $\xi\in\mathbb{R}^2$,
\begin{equation}\label{fss}
\widehat{\mu}(\xi)=\sum_{\bi\in \Lambda^q} X^{q,(1)}_\bi \mathrm{e}^{\ri \pi \langle a_\bi,\xi \rangle} \widehat{T_q\mu^{q,(1)}}(\xi ).
\end{equation}
Iterating \eqref{fss} and taking the limit, 
\begin{equation}\label{fssi}
\widehat{\mu}(\xi)=\prod_{n=0}^\infty \Psi_n^q(\xi),
\end{equation}
where, for $n\ge 0$,
\[
\Psi_n^q(\xi)=\sum_{\bi\in\Lambda^q} X^{q,(n+1)}_\bi \mathrm{e}^{\ri\pi \langle T_q^na_\bi,\xi \rangle}.
\]

From \eqref{fssi}, for $q\ge 1$ and $k\ge 2$, we can write $\mu$ as a convolution of two measures $\mu_{q,k}*\eta_{q,k}$, where
\[
\widehat{\mu_{q,k}}(\xi)=\prod_{ k{\not\hspace{0.07cm}\mid}\,n+1} \Psi_n^q(\xi) \ \text{ and } \ \widehat{\eta_{q,k}}(\xi)=\prod_{k\mid n+1} \Psi_n^q(\xi).
\]
Notice that $\mu_{q,k}$ is within the class of random self-similar measures of the form (\ref{ifsr})--(\ref{mes2}); indeed it has the same law as the random self-similar measure with respect to the IFS
\[
\{T_q^k\cdot+f_{\bi_1}\cdots f_{\bi_{k-1}}((0,0))\}_{ \bi_1\ldots \bi_{k-1}\in (\Lambda^q)^{k-1}}
\]
and the random vector
\[
\{X^{q,(1)}_{\bi_1}\cdots X^{q,(k-1)}_{\bi_{k-1}}\}_{\bi_1\ldots \bi_{k-1}\in (\Lambda^q)^{k-1} }.
\]
Thus, almost surely, $\dim_H \mu_{q,k}=\frac{k-1}{k}\dim_H \mu= \frac{k-1}{k}s >1$ by  \eqref{dim>1} and for   some sufficiently large $k$ which we fix for the remainder of the proof. Applying Theorem \ref{proj} we can find a set $\Omega_1$ with $\mathbb{P}(\Omega_1)=1$ such that, for all $\omega\in \Omega_1$, for all $\beta\in[0,\pi)$, $q\ge 1$,
\begin{equation}\label{omega1}
\dim_H \pi_{\beta} \mu_{q,k}=1.
\end{equation}

The rest of the proof estimates the Fourier transform of $ \pi_{\beta}\eta_{q,k}$ using Lemma \ref{cover}. From \eqref{wij}, for $q\ge1$ and $n\ge 0$ the event
\[
A_{q,n}=\{X_{i}^{(qn+h)}, X_{j}^{(qn+h)}\ge p_* \text{ for some } h=0,\ldots,q-1\}
\]
has probability $\mathbb{P}(A_{q,n})=1-(1-p)^q$. Since $\{\chi_{A_{q,nk}}\}_{n\ge 0}$ are i.i.d. random variables for all $q\ge 1$, by the strong law of large numbers we can find a set $\Omega_2$ with $\mathbb{P}(\Omega_2)=1$ such that for all $\omega\in \Omega_2$, for all $q\ge 1$,
\begin{equation}\label{omega2}
\lim_{N\to\infty} \frac{1}{N}\sum_{n=0}^N \chi_{A_{q,nk}}(\omega)=1-(1-p)^q.
\end{equation}
By \eqref{uni} we may also find a set $\Omega_3$ with $\mathbb{P}(\Omega_3)=1$ such that for all $n\ge 1$,
\begin{equation}\label{omega3}
\text{there exists } \ell \in \Lambda \text{ such that } X_\ell^{(n)}\ge p_*.
\end{equation}

 Take $\omega\in \Omega_1\cap \Omega_2\cap \Omega_3$. The rest of the proof will be deterministic.

Let $\beta\in [0,\pi)\setminus E$. By Lemma \ref{cover} there exists  $i_0=i_0(\beta)$ such that for all $q\ge 1$ there exists $N_0=N_0(\beta,q)$ such that $\beta\not\in E_{q,k}(1/i_0,N)$ for all $N\ge N_0$. In other words, for all $N\ge N_0$,
\begin{equation}\label{lb}
\max_{\tau\in[1,r^{-qk}]}\frac{1}{N}\#\Big\{n\in [N]:\|b \tau r^{q-qk(N-n)}\cos(\beta+\gamma-nqk\theta)\|> \frac{r^{2qk}}{15} \Big\}\geq \frac{1}{i_0},
\end{equation}
where $\gamma=\mbox{ang} (a_{i}-a_{j})+\theta$ and $b=|a_{i}-a_{j}|$. Take $q$ large enough so that $(1-p)^q<1/4i_0$. We show, in a similar manner to \cite[Proposition 3.3]{SS}, that the projected measure $\pi_\beta \eta_{q,k}$ has positive Fourier dimension. (Recall that the Fourier dimension of a measure $\lambda$ is the supremum of $\sigma$ such that $\widehat{\lambda}(\xi) = O(|\xi|^{-\sigma/2})$.)

Writing  $w_\beta=(\cos \beta,\sin \beta)$ as before and applying the formula
\[
\widehat{\pi_\beta \lambda}(t)=\widehat{\lambda}(t w_\beta)\quad (t \in \mathbb{R})
\]
for the Fourier transform of the projection of a measure $\lambda$ on $\mathbb{R}^2$, we obtain
\[
\widehat{\pi_\beta \eta_{q,k}}(t)=\prod_{n=1}^\infty \Psi_{kn-1}^q(t w_\beta).
\]
By \eqref{omega2}, we can find an integer $N_1$ such that for all $N\ge N_1$,
\begin{equation}\label{omega2'}
\sum_{n=0}^N \chi_{A_{q,nk}}(\omega)\ge N(1-2(1-p)^q)\ge N(1-1/2i_0).
\end{equation}
We claim that if $\chi_{A_{q,nk}}(\omega)=1$, then there exist distinct $\bi_1,\bi_2\in \Lambda^q$ such that
\begin{equation}\label{bi12}
X^{q,(kn)}_{\bi_1}, X^{q,(kn)}_{\bi_2}\ge (p_*)^q;\ \ \  \mathrm{ang} (a_{\bi_1}-a_{\bi_2})=\mbox{ang} (a_{i}-a_{j}); \ \ \ |a_{\bi_1}-a_{\bi_2}|=br^{q-1}.
\end{equation}
To see this, by \eqref{omega3} we can find $i_1,\ldots, i_q\in \Lambda$ such that  $X_{i_l}^{(qn+l)}\ge p_*$ for all $l=0,\ldots,q-1$.
Also $X_{i}^{(qkn+h)}, X_{j}^{(qkn+h)}\ge p_*$ for some $h\in\{0,\ldots,q-1\}$ since $\chi_{A_{q,kn}}(\omega)=1$.  
Then it is easy to check that $\bi_1=i_1\ldots i_{h-1} i i_{h+1}\ldots i_q$ and $\bi_2=i_1\ldots i_{h-1} j i_{h+1}\ldots i_q$ satisfy \eqref{bi12}. Hence for all $n\ge 0$ such that $\chi_{A_{q,nk}}(\omega)=1$ we can write, for some $d_0, d_\bi\in \mathbb{R}$,
\begin{eqnarray}\label{psisuma}
\Psi_{nk}^q(tw_\beta)&=& \sum_{\bi\in\Lambda^q} X^{q,(nk)}_\bi \mathrm{e}^{\ri\pi \langle T_q^{nk-1}a_\bi,tw_\beta \rangle}\nonumber\\
&=&
\mathrm{e}^{\ri\pi d_0}\Big(X^{q,(nk)}_{\bi_1}+X^{q,(nk)}_{\bi_2} \mathrm{e}^{\ri\pi \langle T_q^{nk-1}(a_{\bi_2}-a_{\bi_1}), tw_\beta \rangle}+\sum_{\bi \neq \bi_1,\bi_2}X^{q,(nk)}_{\bi} \mathrm{e}^{\ri\pi d_{\bi}}   \Big).\nonumber
\end{eqnarray}

Let $ t=\tau (r^{-qk})^N$, where $\tau\in[1,r^{-qk}]$ and $N\ge N_2:= \max\{N_0, N_1\}$, where $N_0=N_0(\beta,q)$ is given for \eqref{lb}. Note that $\mathrm{ang} (a_{\bi_1}-a_{\bi_2})=\gamma-\theta$. Then
\begin{eqnarray*}
 \langle T_q^{nk-1}(a_{\bi_2}-a_{\bi_1}), tw_\beta \rangle 
&=& b r^{q-1}t \langle T_q^{nk-1}w_{-\gamma+\theta}, w_\beta \rangle\\
&=& b \tau r^{q-qk(N-n)}\cos(\beta+\gamma -nqk\theta).
\end{eqnarray*}
Since $\sum_{\bi\in\Lambda^q} X^{q,(nk)}_{\bi}= 1$ and $X^{q,(nk)}_{\bi_1},X^{q,(nk)}_{\bi_2} \geq (p_*)^q$,  there is a constant $\rho=\rho(p_*,r,q,k)>0$ such that
$$|\Psi_{nk-1}^q(tw_\beta)|\le 1-\rho$$
whenever $\|b \tau r^{q-qk(N-n)}\cos(\beta+\gamma -nqk\theta)\|>r^{2qk}/15$. From \eqref{lb} and \eqref{omega2'} we deduce that
\[
\#\big\{n\in[N]: |\Psi_{nk-1}^q(tw_\beta)|\le 1-\rho \big\}\ge \Big(\frac{1}{i_0}-\frac{1}{2i_0}\Big)N=\frac{1}{2i_0}N.
\]
Hence
\[
|\widehat{\pi_\beta \eta_{q,k}}(t)| \le (1-\rho)^{ N/2i_0} \le t^{-\log(1-\rho)/(2i_0 qk\log r)},
\]
provided $t\ge r^{-qk(N_2+1)}$, so $\pi_\beta \eta_{q,k}$ has positive Fourier dimension.

It was shown in \cite[Lemma 4.3]{SS}  that the convolution of a measure of full Hausdorff dimension with one of positive Fourier dimension is absolutely continuous with respect to Lebesgue measure. Since $\dim_H \pi_\beta \mu_{q,k}=1$ by \eqref{omega1}, applying \cite[Lemma 4.3]{SS} to  $\pi_\beta \mu_{q,k}$ and $\pi_\beta \eta_{q,k}$ gives that $\pi_\beta\mu$ is absolutely continuous.
\end{proof}

We now apply Theorems \ref{prop1} and \ref{absss} to get weak dimension conservation for self-similar sets in $\mathbb{R}^2$ where the IFS consists of similarities with irrational rotations that are translates of each other and satisfy  OSC.

\begin{theorem}\label{thmhd}
Let $\theta/\pi$ be irrational and suppose that the IFS $\mathcal{I}=\{f_i=rR_\theta\cdot+a_i\}_{i=1}^m$ on $\mathbb{R}^2$, with $r>1/m$ and satisfying OSC, has attractor $K$, so that $s:= \dim_H K=-\log m/\log r >1$. 

Then there is a set $E\subset [0,\pi)$ with $\dim_H E =0$ such that for all $\beta\in [0,\pi)\setminus E$, for all $\epsilon\in(0,s-1)$,
$$\mathcal{L}^1\big\{x\in \pi_\beta(K): \dim_H \big(K\cap \pi_\beta^{-1}(x)\big)\ge s-1-\epsilon\big\}>0.$$
\end{theorem}

\begin{proof}
For each integer $q>\log 2/-\log r$ we may regard $K$ as the attractor of the IFS ${\mathcal I}_q := \{f_{i_1}\cdots f_{i_q}: 1\leq i_1,\ldots,i_q \leq m\}$ so that $K = \Phi_q( \Sigma_q)$ where
$\Sigma_q := \{\bi_1\bi_2\ldots: \bi_j \in \Lambda^q\}$ and $\Phi_q$ is the cannonical map. Let $E_q\subset [0,\pi)$ be the set with $\dim_H E_q= 0$ given by that Theorem \ref{absss} for the IFS $\mathcal{I}_q$. Take $E=\cup_{q>\log 2/-\log r} E_q$ so that $\dim_H E=0$.

Now fix $\epsilon\in(0,s-1)$. Let  $q>\log 2/-\log r$ be an integer to be specified later. Let
\begin{equation}\label{probdef}
p_{q}: =\Big(r^{q(s-1-\epsilon)}-\frac{2}{m^q}\Big)\frac{m^q}{m^q-2}
= \frac{m^{q(1+\epsilon)/s}-2}{m^q-2} \in (0,1),
\end{equation}
since $r^s = m^{-1}$ and $2<r^{-q}= m^{q/s}$. Let $S_{q}$ be a random subset of $\Lambda^q$ defined as follows. First choose two different symbols from $\Lambda^q$ with uniform probability, then select each of the remaining $m^q -2$ symbols with probability $p_{q}$, all actions being independent; in this way  $S_{q}$ always contains at least two symbols. Moreover, for each $\mathbf{i}\in \Lambda^q$,
\begin{align*}
\mathbb{P}(\mathbf{i}\in S_{q})&= \frac{2}{m^q}+\frac{m^q-2}{m^q}p_{q} \\
&= r^{q(s-1-\epsilon)}.
\end{align*}
Let $\{S_{q}^{(k)}:k\in\mathbb{N}\}$ be a sequence of independent copies of $S_{q}$. Then the set
\[
\Sigma_{q}^\omega:=S_{q}^{(1)}\times S_{q}^{(2)}\times \cdots 
\] 
is an $\alpha$-random set, with $\alpha=\log r^{q(s-1-\epsilon)}/\log r^q = s-1-\epsilon$, witn $\Phi_q$ satisfying (1) and (2) at (\ref{lambdar}) and (\ref{overlap}). 

Define a random vector $X_{q}$ in a uniform manner, that is,
$$X_{q} =\bigg\{\frac{\chi_{(\bi \in S_{q})}}{ \#S_{q}}\bigg\}_{\bi \in S_{q}};$$
then $(X_{q})_\bi \geq 1/m^q := p_*$  for at least two $\bi \in S_{q}$.
Let $\{X_{q}^{(k)}:k\in\mathbb{N}\}$ be independent copies of $X_{q}$ which are supported by $S_{q}^{(k)}$. These random vectors define  a random measure $\nu_{q}$ on $\Sigma_q$ of the form described in \eqref{ifsr} and \eqref{mes2} at the start of this section. Then $\nu_q$ has support $\Sigma^\omega_q$, and $\Phi_q\nu_{q}$ has support $K^\omega = \Phi_q(\Sigma^\omega_q)$. 
From the strong law of large numbers, and using OSC when mapping the measure under $\Phi_q$, almost surely
\[
\dim_H \Phi_q\nu_{q}=\frac{\mathbb{E}(\log \# S_{q})}{-\log r^q}.
\]
Write $\mbox{Bin}(n,p)$ to denote the binomial distribution with $n$ points and probability $p$. Then   
\[
\mathbb{E}(\log \# S_{q})=\mathbb{E}\big(\log[\mathrm{Bin}(m^q-2,p_{q})+2]\big)= \log(m^{q(1+\epsilon)/s})-o(1)
\]
as $q \to \infty$, on using (\ref{probdef}) to express $p_{q}$ in terms of $m$ together with a simple application of Chebyshev's inequality. Thus
\[
\dim_H \Phi_q\nu_{q}=\frac{\log(m^{q(1+\epsilon)/s})-o(1)}{-\log m^{-q/s}}=1+\epsilon-o(q^{-1})>1
\]
provided we now choose $q$ sufficiently large. 

From Theorem \ref{absss}, almost surely for all $\beta\in [0,\pi)\setminus E \subset [0,\pi)\setminus E_q$, the projected measure $\pi_\beta \Phi_q\nu_{q}$ is absolutely continuous with respect to Lebesgue measure, so $\mathcal{L}^1(\pi_\beta(K^\omega))>0$. The conclusion follows from Proposition \ref{prop1}, taking $A= \Sigma$, $K= \Phi(\Sigma)$and $\alpha = s-1-\epsilon$.
\end{proof}

We now extend Theorem \ref{thmhd} to general sets of similarities using a technique of Peres and Shmerkin \cite[Proposition 6]{PS}. This allows us to reduce a general plane IFS to one where the similarities are mutual translates with the attractor a subset of that of the original IFS and of arbitraily close dimension to which we may apply Theorem \ref{thmhd}.

\begin{proposition}\label{samerots}
Let 
$\mathcal{I}=\{f_i=r_iR_{\theta_i}\cdot+a_i\}_{i=1}^m$
be an IFS on $\mathbb{R}^2$ satisfying OSC with attractor $K$. For all 
$\epsilon >0$ there is an IFS
$\mathcal{I}_\epsilon$, satisfying SSC and formed by a collection of compositions of maps from  $\mathcal{I}$, such that all the maps in $\mathcal{I}_\epsilon$ have the same rotation $R_\theta$ for some angle $\theta$ and the same contraction ratio $0<r<1$, and with attractor $K_\epsilon \subset K$ such that $\dim_H K_\epsilon > \dim_H K - \epsilon$.

Moreover, if $\mathcal{I}$ has dense rotations then we may take $\theta/\pi$ to be irrational.
\end{proposition}

\begin{proof}
First we may assume that $\mathcal{I}$ satisfies SSC, since there is an IFS formed by compositions of the maps in $\mathcal{I}$ that satisfies SSC with attractor a subset of $K$ and with Hausdorff dimension arbitrarily close to that of $K$, see, for example, \cite{Orp}. 

Next, as in the proof of \cite[Proposition 6]{PS}, we may find integers $n_1,\ldots,n_m$ such that the IFS $\mathcal{I}_\epsilon$ formed by all those compositions of the maps of  $\mathcal{I}$ taken in any order such that $f_i$ occurs $n_i$ times for each $i = 1, \ldots,m$, has an attractor $K_\epsilon\subset K$ with $\dim_H K_\epsilon> \dim_H K - \epsilon$. All the maps in $\mathcal{I}_\epsilon$ have rotation $R_\theta= R_{n_1\theta_1+\cdots+n_m\theta_m}$ and contraction ratio $r= r_1^{n_1}\cdots r_m^{n_m}$. 

Now suppose that  $\mathcal{I}$ has dense rotations. If $(n_1\theta_1+\cdots+n_m\theta_m)/\pi$ is irrational then there is nothing further to prove. Otherwise, at least one of the $\theta_i$, say $\theta_1$, is an irrational multiple of $\pi$. By a slight modification of the proof of \cite[Proposition 6]{PS} we may conclude  that the attractor of the IFS $\mathcal{I}'_\epsilon$ formed by the compositions of the maps of  $\mathcal{I}$  such that $f_1$ occurs $n_1 -1$ times and $f_i$ occurs $n_i$ times for $i = 2, \ldots,m$, with attractor $K'_\epsilon\subset K$  has $\dim_H K'_\epsilon > \dim_H K - \epsilon$. (We just note in \cite[Proposition 6]{PS} that the number of paths ending at a neighboring lattice point to $v$ is comparable to the number of paths ending at $v$.) Then 
$((n_1-1)\theta_1+\cdots+n_m\theta_m)/\pi$ is irrational so the conclusion holds for $\mathcal{I}'_\epsilon$.
\end{proof}

\begin{theorem}\label{thmhdgen}
Let 
$$\mathcal{I}=\{f_i=r_iR_{\theta_i}\cdot+a_i\}_{i=1}^m$$
be an IFS on $\mathbb{R}^2$ with dense rotations satisfying OSC, with attractor $K$ and with $s =\dim_H K >1$, where $s$ is given by $\sum_{i=1}^m r_i^s=1$. 
Then there is a set $E\subset [0,\pi)$ with $\dim_H E =0$ such that for all $\beta\in [0,\pi)\setminus E$, for all $\epsilon\in(0,s-1)$,
\begin{equation}\label{ineq1}
\mathcal{L}^1\big\{x\in \pi_\beta(K): \dim_H \big(K\cap \pi_\beta^{-1}(x)\big)\ge s-1-\epsilon\big\}>0.
\end{equation}
\end{theorem}

\begin{proof}
For each $\epsilon>0$, applying  Theorem \ref{thmhd} to the amended  IFS $\mathcal{I}_\epsilon$ with attractor $K_\epsilon$ given by Proposition \ref{samerots} (replacing $\epsilon$ by $\epsilon/2$ in both theorem and proposition),
there is a set $E_\epsilon \subset[0,\pi)$ with $\dim_H E_\epsilon = 0$, such that \eqref{ineq1} holds for all $\beta \in E_\epsilon$.
So that the set of exceptional $\beta$ does not depend on $\epsilon$, we let $E=\cup_{n=n_0}^\infty E_{2^{-n}}$, where $2^{-n_0} < s-1$, so that $\dim_H E= 0$.
\end{proof}

\section{Further remarks}
\setcounter{theorem}{0}

\noindent 1. A natural question is whether these results can be strengthened from `weak dimension conservation' to `dimension conservation', that is whether the `$\epsilon$' can be removed in  the conclusion of Proposition \ref{prop2},
and in Theorems \ref{dimconthm}, \ref{mandper},  \ref{thmhd} and \ref{thmhdgen}. 
\medskip

\noindent 2. Another natural question is whether, in Proposition \ref{prop2}, the condition on the projection of $B(\Phi[\bi])$ can be weakened, with a consequential weakening of the corresponding condition on the projections of $K$ in Theorem \ref{dimconthm}. 
Furthermore, can $\lbd$ of the sections be replaced by $\dim_H$  in  the conclusions of Proposition \ref{prop2} and Theorem \ref{dimconthm}?
An alternative approach would be to eliminate the exceptional set of directions in Theorem \ref{thmhd} and thus Theorem \ref{thmhdgen}.

This raises the question of whether the box-dimension and Hausdorff dimension of sections of self-similar set are `typically' equal for all, or perhaps `nearly all' directions. If $\lbd (K\cap L )=\dim_H (K\cap L )$ for every line $L$, or at least for a large set of lines, then one might be able to replace lower box dimension by Hausdorff dimension in the conclusion of  Theorem \ref{dimconthm}. There are plane self-similar sets defined by homotheties with at least some sections having distinct Hausdorff and lower box dimensions, for example for certain horizontal sections of the 1-dimensional Sierpi\'{n}ski triangle, that is the attractor of the plane IFS with maps $f_1(x,y) = (\frac{1}{3}x, \frac{1}{3}y),  f_2(x,y) = (\frac{1}{3}x+\frac{2}{3}, \frac{1}{3}y), f_3(x,y) = (\frac{1}{3}x, \frac{1}{3}y+\frac{2}{3})$ (we are grateful to Thomas Jordan for pointing out this example to us); see also \cite{BFS}. Is this possible for self-similar sets with dense rotations?
 \medskip

\noindent 3. Similar conclusions to Proposition \ref{prop2} and thus Theorem \ref{dimconthm}   might be expected for projections onto $k$-dimensional subspaces $V$ where $k\geq 2 $. However,  it seems hard to get an analogue of Lemma \ref{sub1} in this case. One would need to show that for any cube $I \subset V$ with $|I| \leq r$ there is a bounded number of points $x_i \in V$ with $N(x_i,r) \leq M$ such that if $N(x,r)\leq M$ for some $ x \in I$  then                            
 some $L_{x_i}$ intersects every set $B( \Phi[\bi])$ such that $\bi \in \Lambda_r$ that intersects $L_x$. (Here  $N(x,r)$ is the number of $B( \Phi[\bi])$ with $\bi \in \Lambda_r$ that intersect $L_x$, the $(d-k)$-plane through $x \in V$ and perpendicular to $V$.)
  \medskip
  
 4. Our results have been presented for self-similar sets defined by orientation-preserving similarities. It would be possible to extend them to allow some of the maps to be 
 orientation-reversing, for example by replacing an IFS by one formed by appropriate orientation-preserving compositions of the maps with little reduction in the dimension of the attractor, as in the proof of \cite[Proposition 6]{PS}.

 \bibliographystyle{abbrv}

\end{document}